\documentclass[11pt, a4paper]{amsart}

\usepackage{asymptote}
\usepackage{xcolor}
\usepackage{hyperref}
\hypersetup{
  hidelinks,
  pdftitle={Geometric Large-Deviation-Type Principles for Mixed Measures},
  pdfauthor={Malak Lafi and Artem Zvavitch},
  pdfkeywords={mixed measures, log-concavity, large deviations}
}
\usepackage{float}
\usepackage{tikz}
\usetikzlibrary{arrows.meta,calc}

\usepackage[font=small,labelfont=bf]{caption}
\usepackage{amsmath,amsthm,amssymb}
\usepackage{mathtools}
\mathtoolsset{showonlyrefs}
\makeatletter
\renewcommand\subsection{\@startsection{subsection}{2}{\z@}%
  {-3.25ex\@plus -1ex \@minus -.2ex}%
  {1.5ex \@plus .2ex}%
  {\normalfont\large\bfseries}}
\makeatother

\title{Geometric Large-Deviation-Type Principles for Mixed Measures}
\author{Malak Lafi}
\address{Department of Mathematical Sciences, Kent State University, Kent, OH 44242, USA}
\email{mlafi1@kent.edu}
\author{Artem Zvavitch}
\address{Department of Mathematical Sciences, Kent State University, Kent, OH 44242, USA}
\email{azvavitc@kent.edu}
\thanks{Both authors are supported in part by the U.S. National Science Foundation Grant DMS-439984 and the United States--Israel Binational Science Foundation (BSF) Grant 2018115.}

\date{}

\begin{document}

\begin{abstract} We study a geometric analogue of the large deviation principle for mixed measures associated with a class of $\log$-concave probability measures whose densities depend on the gauge of a convex body. For convex bodies in $\mathbb{R}^n$, we prove a geometric large-deviation-type asymptotic for first-order mixed measures, in which the decay under dilation is governed by a natural inradius associated with the measure. In the planar case, we derive an explicit representation
and prove a genuine logarithmic limit for second-order mixed
measures.  As an application, we prove a comparison theorem showing that asymptotic dominance under dilation forces inclusion between convex bodies.
\end{abstract}
\numberwithin{equation}{section}
\newtheorem{theorem}{Theorem}[section]
\newtheorem{lemma}[theorem]{Lemma}

\newtheorem{remark}[theorem]{Remark}

\maketitle
\tableofcontents

\section{Introduction}

Mixed measures extend the variational structure of the classical Brunn--Minkowski theory from Lebesgue measure to more general measures. They arise as first and higher variations of a measure under Minkowski addition and therefore encode an interaction between the geometry of convex bodies and the decay of the underlying density.

In this paper, we study the behavior under dilation of mixed measures associated with a class of $\log$-concave measures whose densities depend on the gauge of a convex body. Our main question is whether the exponential decay of such mixed measures is governed by a natural geometric parameter, in the same way that the decay of the measure of a dilated complement is governed by an inradius in the large-deviation-type principle obtained in \cite{LZ}.

Let $L\subset\mathbb R^n$ be a convex body containing the origin in its interior, let $\phi:[0,\infty)\to[0,\infty)$ be nondecreasing, nonconstant, and convex, and let $\mu$ be the log-concave probability measure on $\mathbb{R}^n$ whose density is proportional to
$
e^{-\phi(\|x\|_L)}.$

For a convex body $K$ containing the origin in its interior, define its centered $L$-inradius by
\[
r(K,L)=\max\{R>0:RL\subset K\}.
\]
Our first main result shows that, for convex bodies $K,M\subset\mathbb R^n$ containing the origin in their interiors,
\[
\limsup_{t\to\infty}
\frac{\ln\mu(tK;M)}{\phi(r(K,L)t)}=-1.
\]
Moreover, for every $\varepsilon>0$, the set of parameters $t$ for which the normalized logarithm differs from $-1$ by at least $\varepsilon$ has finite Lebesgue measure. Thus the centered $L$-inradius determines the geometric large-deviation-type rate of the first mixed measure.

The proof separates naturally into two parts. The first is geometric: near a contact point between $K$ and $r(K,L)L$, we obtain a polynomial lower bound for the size of the portion of $\partial K$ contained in a slightly larger dilate of $L$. The second is analytic: we establish a one-dimensional Laplace-type estimate for the integrals resulting from this geometric bound. We formulate these two ingredients as separate lemmas in Section~\ref{twobodies}.

As applications, we recover the large-deviation-type principle for $\mu((tK)^c)$ from \cite{LZ} and prove a comparison theorem: if
\[
\mu(tRL;M)\ge\mu(tK;M)
\]
for all sufficiently large $t$, then $RL\subseteq K$.

We also study the second mixed measure $\mu(A;B,C)$. In the planar case, we derive an explicit integral representation for measures with densities of the form $e^{-\phi(\|x\|_L)}$. Under suitable smoothness assumptions and the additional condition,
\[
\lim_{t\to\infty}\frac{\ln\phi'(t)}{\phi(t)}=0,
\]
we prove that the following limit exists and equals $-1$:
\[
\lim_{t\to\infty}
\frac{\ln[-\mu(tA;B,C)]}{\phi(r(A,L)t)}=-1.
\]
The Gaussian case is treated separately to illustrate the geometric mechanism of the argument.

We use the term large-deviation-type principle in a geometric, Laplace-type sense: it refers to exponential asymptotics under dilation rather than to a classical large deviation principle on a fixed probability space. For background on classical large-deviation principles and Laplace-type asymptotics, we refer to \cite{DZ} and \cite{Bruijn}, respectively.

The paper is organized as follows. Section~\ref{preliminaries} collects the notation and background on convex bodies, mixed volumes, and mixed measures. In Section~\ref{twobodies}, we prove the large-deviation-type principle for first mixed measures. Section~\ref{comparison} contains the applications to complements and comparison problems. Finally, Section~\ref{three} treats second mixed measures in the plane.

\section{Background and notation}\label{preliminaries}

We begin with notation and several standard definitions.  Throughout, $n\ge2$ unless explicitly stated otherwise. We denote by $\mathbb R^n$ the $n$-dimensional Euclidean space. The inner product of $x,y\in\mathbb R^n$ is denoted by $\langle x,y\rangle$, and $|\cdot|$ denotes the Euclidean norm. We write $B_2^n$ for the closed Euclidean unit ball and $\mathbb S^{n-1}=\partial B_2^n$ for the unit sphere. A convex body is a compact convex subset of $\mathbb R^n$ with nonempty interior. The support function of a compact convex set $K$ is
\[
h_K(x)=\sup_{y\in K}\langle x,y\rangle.
\]
If $0\in\operatorname{int}K$, its Minkowski functional (or gauge) is
\[
\|x\|_K=\inf\{\lambda\ge 0:x\in\lambda K\}.
\]
For arbitrary subsets $A,B\subset\mathbb R^n$, their Minkowski sum is $A+B=\{a+b:a\in A,\ b\in B\}$, and $\alpha A=\{\alpha x:x\in A\}$ for $\alpha\in\mathbb R$. For a measurable set $A\subset\mathbb R^n$, $|A|$ denotes its $n$-dimensional Lebesgue measure, while $\mathcal H^k$ denotes the $k$-dimensional Hausdorff measure. We denote the interior and boundary of a set $A$ by $\operatorname{int}A$ and $\partial A$, respectively. For a convex body $K$ and a linear subspace $H$, we denote the orthogonal projection of $K$ onto $H$ by $P_HK$.

Recall that a function $f$ is convex on a convex set $C$ if
\[
f(\lambda x+(1-\lambda)y)\le \lambda f(x)+(1-\lambda)f(y)
\]
for all $x,y\in C$ and $\lambda\in[0,1]$; it is concave if $-f$ is convex. A Borel measure $\mu$ on $\mathbb R^n$ is called log-concave if
\[
\mu(\lambda A+(1-\lambda)B)\ge \mu(A)^\lambda\mu(B)^{1-\lambda}
\]
for all Borel sets $A,B\subset\mathbb R^n$ and $\lambda\in(0,1)$. Borell's classification \cite{Borell} implies that a log-concave Radon measure is supported on an affine subspace and has a log-concave density with respect to Lebesgue measure on that subspace. In the full-dimensional case, this density has the form $e^{-\psi}$, where $\psi:\mathbb R^n\to(-\infty,\infty]$ is convex. Our main examples have densities proportional to $e^{-\phi(\|x\|_L)}$, where $L$ is a convex body with $0\in\operatorname{int}L$ and $\phi$ is nondecreasing and convex.

\subsection{Classical mixed volumes}
Mixed volumes play a central role in understanding the geometric interaction between convex bodies in $\mathbb R^n$. We briefly recall some classical facts and definitions that will be used throughout this paper; further details can be found in Schneider's monograph \cite{Schneider}. Let $K_1,\dots,K_m$ be compact convex subsets of $\mathbb R^n$, and let $t_1,\dots,t_m\ge0$. Minkowski's theorem on mixed volumes asserts that the volume of the Minkowski linear combination $t_1K_1+\dots+t_mK_m$ is a homogeneous polynomial of degree $n$ in the coefficients $t_1,\dots,t_m$; namely,
\begin{equation} \label{eq:mixed-volume-polynomial}
    |t_1K_1+ \dots +t_mK_m|= \sum _{i_1,\dots, i_n=1}^{m} t_{i_1} \dots t_{i_n} V(K_{i_1}, \dots, K_{i_n}).
\end{equation}
The coefficients $V(K_{i_1},\dots,K_{i_n})\ge0$ are invariant under permutations of their arguments and are called the \textit{mixed volumes} of the $n$-tuple $(K_{i_1},\dots,K_{i_n})$. A particularly important instance of \eqref{eq:mixed-volume-polynomial} is Steiner's formula for two compact convex sets $K,L\subset\mathbb R^n$:
 \begin{equation*} 
     |K+ t L|= \sum _{j=0}^{n} \binom{n}{j} t^j V(K [n-j], L[j]),
 \end{equation*}
where $V(K[n-j],L[j])$ denotes the mixed volume involving $n-j$ copies of $K$ and $j$ copies of $L$. The term with $j=1$, namely $V(K[n-1],L)$, is often referred to as the \textit{first mixed volume} of $K$ and $L$. It admits the well-known limit representation
  \begin{equation} \label{eq:first-mixed-volume}
     V(K [n-1], L) = \frac{1}{n} \lim _{\varepsilon \to 0^+} \frac{|K+ \varepsilon L| -|K|}{\varepsilon}=\frac{1}{n} \frac{d}{d\varepsilon} |K+\varepsilon L| \Bigr |_{\varepsilon =0}.
 \end{equation}
 In particular, the surface area of $K$ can be computed as  $\mathcal H^{n-1}(\partial K)=nV(K[n-1],B_2^n).$
This representation highlights the variational nature of mixed volumes and provides the foundation for their generalization to arbitrary measures, which we discuss next.
\subsection{From mixed volumes to mixed measures}
The concept of mixed volume extends naturally from Lebesgue measure to more general Borel measures on $\mathbb R^n$. This extension was developed by Livshyts \cite{LG}, who introduced \textit{mixed measures} to capture the first-order variation of a measure under Minkowski addition. Given a Borel measure $\mu$ and Borel sets $K,L\subset\mathbb R^n$, the \textit{mixed measure} of $K$ and $L$ is defined by
\begin{equation} \label{eq:mixed-measure-definition}
    \mu (K; L):=\liminf _{\varepsilon\to 0^+}\frac{\mu (K+\varepsilon L)- \mu (K)}{\varepsilon}.
\end{equation}
When $\mu$ is Lebesgue measure, equation \eqref{eq:mixed-measure-definition} gives exactly
$\mu(K;L)=nV(K[n-1],L)$.

A particular case of mixed measures arises when the second body is the Euclidean ball. For a Borel measure $\mu$ on $\mathbb{R}^n$ and a convex body $K$, the quantity
\[
\mu^+(\partial K):=\liminf_{\varepsilon\to0^+}
\frac{\mu(K+\varepsilon B_2^n)-\mu(K)}{\varepsilon}
\]
is known as the \emph{outer Minkowski content} of $K$ with
respect to $\mu$. For the classical theory of Minkowski content,
see \cite[Sections~3.2.34--3.2.40]{Federer}; for the weighted
setting, see \cite{LG,WBMTI}. When $\mu$ has a continuous density,
this $\liminf$ is in fact a limit.
 
To obtain an integral representation of mixed measures, we first recall the surface area measure. For a convex body $K\subset\mathbb R^n$, the surface area measure $S_K$ is the Borel measure on $\mathbb S^{n-1}$ defined through the Gauss map $n_K:\partial K\to\mathbb S^{n-1}$, which assigns to almost every point of $\partial K$ its unique outer unit normal. More precisely, for every Borel set $\omega\subset\mathbb S^{n-1}$,
$S_K(\omega)=\mathcal H^{n-1}(n_K^{-1}(\omega))$. If $K$ belongs to the class $C_+^2$, that is, if $\partial K$ is of class $C^2$ and has strictly positive Gauss curvature, then $S_K$ is absolutely continuous with respect to spherical Lebesgue measure, and its density is the curvature function $f_K$; thus $dS_K(u)=f_K(u)\,du$.
 Following \cite{LG,WBMTI}, when $\mu$ has a continuous density $\rho$, the corresponding weighted surface area measure $S_K^\mu$ is defined by
\begin{equation}\label{eq:weighted-area-measure}
\int_{\mathbb S^{n-1}}f(u)\,dS_K^\mu(u)
:=\int_{\partial K}f(n_K(y))\rho(y)\,d\mathcal H^{n-1}(y),
\end{equation}
for every $f\in C(\mathbb S^{n-1})$. If $K$ is a convex body and $L$ is a compact convex set, then the first variation formula gives \cite{LG,WBMTI}
\begin{equation}
\mu(K;L)=\int_{\mathbb S^{n-1}}h_L(u)\,dS_K^\mu(u).
\end{equation}
Equivalently,
\begin{equation}\label{eq:first-variation}
\mu(K;L)=\int_{\partial K}h_L(n_K(y))\rho(y)\,d\mathcal H^{n-1}(y).
\end{equation}
We also use the three-body mixed measure introduced in \cite{WBMTI}. Whenever the mixed right derivative exists, it is defined by
\[
\mu(A;B,C)=\left.\frac{\partial^2}{\partial s\,\partial t}
\mu(A+sB+tC)\right|_{s=t=0^+}.
\]
The following representation theorem will be used later to derive an explicit formula for three-body mixed measures in the planar setting.
\begin{theorem}[{\cite[Theorem~2.7]{WBMTI}}]\label{theorem:WBMTI}
Let $n\ge2$, let $\mu$ have a $C^2$ density $\rho$, let $A\in C_+^2$, and let $B,C$ be compact convex sets. Then
\begin{equation}\label{mimeas3}
\mu(A;B,C)=(n-1)\int_{\mathbb S^{n-1}}h_C(u)\,dS_{A;B}^\mu(u),
\end{equation}
where
\begin{align}\label{defmeas}
dS_{A;B}^\mu(u)
&=\rho(n_A^{-1}(u))\,dS_{A[n-2],B[1]}(u)\nonumber\\
&\quad+\frac1{n-1}\left\langle\nabla\rho(n_A^{-1}(u)),\nabla h_B(u)\right\rangle dS_A(u).
\end{align}
\end{theorem}

For the planar formulas below, set $u(\theta)=(\cos\theta,\sin\theta)$ and, for a support function $h_K$ on $\mathbb S^1$, write $h_K(\theta)=h_K(u(\theta))$; primes denote derivatives with respect to $\theta$ whenever they exist. If $K\in C_+^2$, write $f_K(\theta)=h_K(\theta)+h_K''(\theta)$ and $x_K(\theta)=\nabla h_K(u(\theta))$. 

It was shown in \cite[Section~2.5]{WBMTI} that, for the Gaussian measure on $\mathbb R^2$ with density $d\gamma_2(x)=\frac{1}{2\pi}e^{-|x|^2/2}\,dx$, the mixed measure $\gamma_2(A;B,C)$ is  given by 
\begin{equation}\label{eq:gaussian-planar}
\begin{aligned}
\gamma_2(A;B,C)
&=\frac{1}{2\pi}\int_0^{2\pi}
 e^{-\frac{h_A(\theta)^2+h_A'(\theta)^2}{2}}\\
&\quad\times\Big[h_B(\theta)h_C(\theta)
\big(1-h_A(\theta)f_A(\theta)\big)
-h'_B(\theta)h'_C(\theta)\Big] \,d\theta.
\end{aligned}
\end{equation}   
In Theorem~\ref{theorem_mu}, we generalize \eqref{eq:gaussian-planar} to planar log-concave measures of the form $d\mu(x)=e^{-\phi(\|x\|_L)}\,dx$.

\section{First-Order Mixed Measures}\label{twobodies}

Throughout Sections~\ref{twobodies} and~\ref{comparison}, let $L\subset\mathbb R^n$ be a convex body with $0\in\operatorname{int}L$, let $\phi:[0,\infty)\to[0,\infty)$ be nondecreasing, nonconstant, and convex, and assume that
\[
\int_{\mathbb R^n}e^{-\phi(\|x\|_L)}\,dx=1.
\]
We write
\[
d\mu(x)=e^{-\phi(\|x\|_L)}\,dx.
\]
We impose this normalization only for convenience; a fixed normalization factor may equivalently be absorbed into $\phi$ as an additive constant and does not affect any logarithmic asymptotic below.

We note that, whenever a quotient with denominator $\phi(ct)$, where $c>0$, occurs in a
finite-measure assertion below, it is understood for those $t>0$ for which
$\phi(ct)>0$. The omitted initial interval is bounded and therefore does not
affect the finite-measure conclusion.

We first establish a geometric large-deviation-type principle in the reference case in which the second body is the Euclidean ball $B_2^n$. This setting captures the essential asymptotic mechanism and provides the foundation for the extension to an arbitrary second body. Theorem~\ref{thm:first-mixed-ball} may be viewed as a geometric analogue of the large-deviation-type principle for surface area, since $\mu^+(\partial K)=\mu(K;B_2^n)$. The two lemmas below isolate, respectively, the geometric cap estimate and the one-dimensional analytic estimate used in the proof.
\begin{theorem} \label{thm:first-mixed-ball}
    Let $K\subset {\mathbb R}^n$ be a convex body containing the origin in its interior. Then
\[
\limsup_{t\to\infty}\frac{\ln\mu(tK;B_2^n)}{\phi(r(K,L)t)}=-1.
\]
Moreover, for every $\varepsilon>0$, the set
\[
\left\{t>0:
\left|\frac{\ln\mu(tK;B_2^n)}{\phi(r(K,L)t)}+1\right|\ge\varepsilon\right\}
\]
has finite Lebesgue measure.
\end{theorem}

The proof of Theorem~\ref{thm:first-mixed-ball} uses the following two auxiliary lemmas. The first provides the geometric estimate near a contact point of $K$ and its largest centered copy of $L$.

\begin{lemma}[Geometric cap estimate]\label{lemma:geometric-cap}
Let $K,L\subset\mathbb R^n$ be convex bodies containing the origin in their interiors, and set $R=r(K,L)$. Then there exist constants $P>R$ and $c>0$ such that
\[
\mathcal H^{n-1}(\partial K\cap sL)
\ge c(s-R)^{n-1}
\qquad\text{for every }s\in[R,P].
\]
\end{lemma}

\begin{proof}
By the maximality of $R$ and compactness, $\partial K\cap\partial(RL)\ne\varnothing$. Choose
$w\in\partial K\cap\partial(RL)$. Since $RL\subset K$, there is a supporting hyperplane common to $K$ and $RL$ at $w$; let $\nu$ be an outer unit normal to this hyperplane. Thus
\[
H=\{x:\langle x,\nu\rangle=h_K(\nu)=Rh_L(\nu)\}.
\]
Set $\theta=w/|w|$. Since $0\in\operatorname{int}K$,
$\langle\theta,\nu\rangle=h_K(\nu)/|w|>0$.

Choose $\beta>0$ so small that $\beta\le 1/(2R)$ and
$\beta(K\cap\theta^\perp)\subset\frac12(L\cap\theta^\perp)$, and set
\[
L'=\operatorname{conv}\bigl(\{w/R\}\cup\beta(K\cap\theta^\perp)\bigr).
\]
Then $L'\subset L$ and $RL'\subset RL\subset K$. Put $P=1/\beta\ge2R$; the base of $PL'$ is $K\cap\theta^\perp$.
 For $\alpha\in[R,P]$, set
\[
H_\alpha=H\cap\alpha L',\qquad
\Gamma_\alpha=\partial K\cap\alpha L'.
\]
\begin{figure}[H]
\centering
\begin{asy}
size(12cm);
pair A = (-2,0);
pair B = (3,0);
pair C = (0,2);
draw((-8,0)--(11,0), dashed);
draw((-6,0)--(0,6)--(9,0)--cycle, black);
draw((-3,0)--(0,3)--(4.5,0)--cycle, red);
draw((-2,0)--B--C--cycle, blue);
draw((0,0)--(0,2), blue+dashed);
draw((0,2)--(0,3), red+dashed);
draw((0,3)--(0,7.5), dashed, Arrow);
draw((-6,0) .. (0,2) .. (9,0), linewidth(1.2)+green);
draw((-4.5,1.5)--(36/7, 18/7), linewidth(1.2)+black);
draw((-9/8,15/8)--(9/7, 45/21), linewidth(1.2)+red);
dot((0,2),blue);
dot((0,3), red);
dot((0,6));
label("\tiny $R\beta(K\cap \theta^\perp)$", (0,0),S, blue);
label("\tiny $K\cap \theta^\perp$", (7,0),S);
label("\tiny $\alpha\beta(K\cap \theta^\perp)$", (4,0),S, red);
label("\tiny $\theta^\perp$", (10.5,0),S);
label("\tiny $\partial K$", (6,1.4),S, green);
label("\tiny $RL'$", (1.5,1),S, blue);
label("\tiny $w$", (-1/3,2.5),S,red);
label("\tiny $\frac{\alpha}{R}w$", (-1/6,3.5),SW,red);
label("\tiny $\frac{P}{R}w$", (-1/6,6),W,red);
label("\tiny $H_P$", (3,2.7));
label("\tiny $H_\alpha$", (0.56,2.4),red);
label("\tiny $PL'$", (2,4.5),S);
label("\tiny $\alpha L'$", (1.5,1.9),S,red);
label("\tiny $\vec{\theta}$", (0,7),E);
\end{asy}
\caption{ Intersections
of the cones $\alpha L'$ and $P L'$ with the supporting hyperplane $H$ and
the boundary of $K$.}
\label{levelset}
\end{figure}
 The slices $H_\alpha$ and $H_P$ are homothetic with center $w$ and ratio $(\alpha-R)/(P-R)$. To verify equality rather than only one inclusion, put $v=w/R$ and $B=\beta(K\cap\theta^\perp)$. Since $\langle b,\nu\rangle\le\beta h_K(\nu)\le h_L(\nu)/2$ for $b\in B$, write a point of $\alpha L'=\operatorname{conv}(\{\alpha v\}\cup\alpha B)$ as $x=\alpha[(1-s)v+sb]$. The condition $x\in H$ gives
\[
s=\frac{(\alpha-R)h_L(\nu)}{\alpha[h_L(\nu)-\langle b,\nu\rangle]}.
\]
Indeed, $h_L(\nu)-\langle b,\nu\rangle\ge(1-\beta R)h_L(\nu)$ and $\alpha\beta\le1$, so the preceding expression belongs to $[0,1]$. Consequently,
\[
H_\alpha=w+(\alpha-R)D,
\qquad
D=\left\{v+\frac{h_L(\nu)}{h_L(\nu)-\langle b,\nu\rangle}(b-v):b\in B\right\}.
\]
The same formula with $\alpha=P$ gives $H_P=w+(P-R)D$, proving the asserted homothety. Hence
\[
\mathcal H^{n-1}(H_\alpha)
=\left(\frac{\alpha-R}{P-R}\right)^{n-1}
\mathcal H^{n-1}(H_P).
\]
For $x\in H_\alpha$, let $z=P_{\theta^\perp}x$, the point $z$ belongs to the base of $\alpha L'$ and, since $\alpha\le P$, also to $K\cap\theta^\perp$. Moreover, $[x,z]\subset\alpha L'$. Thus the segment $[x,z]$ meets $\Gamma_\alpha$, and consequently
$P_{\theta^\perp}H_\alpha\subset P_{\theta^\perp}\Gamma_\alpha$.
Since orthogonal projection is $1$-Lipschitz,
\begin{align}
\mathcal H^{n-1}(\Gamma_\alpha)
&\ge\mathcal H^{n-1}(P_{\theta^\perp}\Gamma_\alpha)
\ge\mathcal H^{n-1}(P_{\theta^\perp}H_\alpha)\nonumber\\
&=\left(\frac{\alpha-R}{P-R}\right)^{n-1}
\mathcal H^{n-1}(H_P)|\langle\theta, \nu\rangle|.
\label{eq:cap-projection}
\end{align}
\[
\mathcal H^{n-1}(\partial K\cap\alpha L)
\ge \mathcal H^{n-1}(\Gamma_\alpha)
\ge
\frac{\mathcal H^{n-1}(H_P)|\langle\theta,\nu\rangle|}
{(P-R)^{n-1}}(\alpha-R)^{n-1}.
\]
This proves the lemma.
\end{proof}

The second auxiliary lemma is the one-dimensional Laplace-type estimate used after the geometric reduction.

\begin{lemma}[One-dimensional Laplace estimate]\label{lemma:laplace-integrals}
Let $0<R<P$, and for every positive integer $m$ define
\begin{equation}\label{Fm}
F_m(t)=\int_{tR}^{tP}
\frac{(y-tR)^{m-1}}{(P-R)^m}e^{-\phi(y)}\,dy.
\end{equation}
Then
\[
\limsup_{t\to\infty}\frac{\ln F_m(t)}{\phi(Rt)}\ge-1.
\]
Moreover, for every $\alpha>1$, the set
\[
\left\{t>0:\frac{\ln F_m(t)}{\phi(Rt)}\le-\alpha\right\}
\]
has finite Lebesgue measure.
\end{lemma}

\begin{proof} First observe that convexity, nonconstancy, and monotonicity imply that $\phi(s)\ge as-b$ for some $a>0$, some $b\in\mathbb R$, and all sufficiently large $s$. Consequently, for every $c>0$,
\begin{equation}\label{eq:lim}
\lim_{t\to\infty}\frac{\ln t}{\phi(ct)}=0.
\end{equation}

We first note that every $F_i$ tends to zero and is eventually decreasing. Indeed,
$F_i(t)\le t^i e^{-\phi(tR)}/i\to0$.  For $i=1$,
\[
F_1'(t)=\frac{Pe^{-\phi(tP)}-Re^{-\phi(tR)}}{P-R},
\]
whereas, for $i\ge2$,
\begin{equation}\label{derivin}
F_i^{\prime}(t)=\frac{P}{P-R}t^{i-1}e^{-\phi(tP)}
-\frac{R(i-1)}{P-R}\int_{tR}^{tP}
\frac{(y-tR)^{i-2}}{(P-R)^{i-1}}e^{-\phi(y)}\,dy.
\end{equation}
For $i=1$, $\phi(tP)-\phi(tR)\to\infty$ gives $F_1'(t)<0$ for large $t$. For $i\ge2$, choose
$R<Q_1<Q_2<P$. Integration over $[tQ_1,tQ_2]$ gives
$F_{i-1}(t)\ge c_i t^{i-1}e^{-\phi(tQ_2)}$. In the derivative formula \eqref{derivin}, this term dominates the endpoint term $t^{i-1}e^{-\phi(tP)}$, because
$\phi(tP)-\phi(tQ_2)\to\infty$. Hence $F_i'(t)<0$ for all sufficiently large $t$.
 Increasing $T$ if necessary, we may also assume that $0<F_i(t)<1$ for every $t\ge T$ and every index under consideration. Consequently, for every $\alpha>1$,
\[
0\le\int_E-\frac{F_i'(t)}{F_i(t)^{1/\alpha}}\,dt
\le\int_T^\infty-\frac{F_i'(t)}{F_i(t)^{1/\alpha}}\,dt
=\frac{\alpha}{\alpha-1}F_i(T)^{1-1/\alpha}<\infty
\]
for every measurable $E\subset[T,\infty)$.

Consider the base case $m=1$. Since
\[
F_1(t)=\frac{1}{P-R}\int_{tR}^{tP}e^{-\phi(y)}\,dy,
\]
the monotonicity of $\phi$ gives $e^{-\phi(y)}\le e^{-\phi(tR)}$ for every $y\ge tR$, and hence
\begin{equation}\label{ineqf1}F_1 (t)\le \frac{1}{P-R}\int_{tR}^{tP} e^{-\phi (tR)}\,dy= te^{-\phi (tR)}.
\end{equation}
Consequently, using \eqref{eq:lim}, we obtain
\begin{equation} \label{F1}
    \limsup_{t\to\infty} \frac{ \ln F_1 (t) }{\phi (Rt)} \le \limsup_{t\to\infty} \frac{ \ln{t}-\phi (Rt)} {\phi (Rt)}=-1.
\end{equation}
      To obtain the reverse inequality, suppose, toward a contradiction, that there exists $\alpha >1$ such that 
    \[
\limsup_{t\to\infty} \frac{\ln F_1 (t)}{\phi (Rt)} < -\alpha.
\] Thus, there exists $t_0>0$ such that for all $t >t_0$ we have
    \begin{equation} \label{cond1}
        F_1(t)^{\frac{1}{\alpha}} \le e^{-\phi(tR)}.
    \end{equation}
Note that 
\[
F_1^{\prime} (t)= \frac{P e^{-\phi (tP)}-R e^{-\phi (tR)}}{P-R},
\] 
thus
\begin{equation}\label{eq:derivative}
e^{-\phi (tR)}=-\frac{P-R}{R} F_1^{\prime}(t) +\frac{P}{R} e^{-\phi (tP)}.
\end{equation}
Substituting \eqref{eq:derivative} into \eqref{cond1}, we obtain
\begin{equation} \label{cond2}
   1 \le -\frac{P-R}{R} \frac{F_1^{\prime}(t)}{F_1(t)^{\frac{1}{\alpha}}} + \frac{P}{R} \frac{e^ {-\phi(tP)}}{F_1(t)^{\frac{1}{\alpha}}}. 
\end{equation}
Using the monotonicity of $\phi$ and inequality \eqref{ineqf1}, we also have
\[
F_1(t) \ge te^{-\phi (tP)}.
\]
Thus, 
  \begin{equation}\label{lower} \frac{e^{-\phi (tP)}}{F_1(t)^{\frac{1}{\alpha}}} \le t^{\frac{-1}{\alpha}} e^ {-(1-\frac{1}{\alpha}) \phi(tP) }.\end{equation}
Substituting \eqref{lower} into \eqref{cond2} gives, for every $t>t_0$,
   \begin{equation}\label{ineqder}
   1 \le -\frac{P-R}{R} \frac{F_1^{\prime} (t)}{F_1(t)^{\frac{1}{\alpha}}} + \frac{P}{R} t^{\frac{-1}{\alpha}} e^ {-(1-\frac{1}{\alpha}) \phi(tP) }.
   \end{equation}
Since $\phi$ is convex and nonconstant, the second term on the right-hand side of \eqref{ineqder} is integrable on $[t_0,\infty)$. The preceding eventual-monotonicity estimate shows that the integral of the first term is finite as well. Integrating \eqref{ineqder} over $[t_0,\infty)$ therefore gives a contradiction. Thus the assumption was false, and together with \eqref{F1} we conclude
\begin{equation} \label{eq:F1-rate}
    \limsup_{t\to\infty} \frac{ \ln F_1 (t) }{\phi (Rt)}=-1.
\end{equation}
The preceding calculation also gives a stronger property. Fix $\alpha>1$ and consider the set
 \begin{equation}\label{omega1}
 \Omega_1=\left\{t\ge T:\frac{\ln F_1(t)}{\phi(Rt)}\le-\alpha\right\}.
 \end{equation}
Then $|\Omega_1|<\infty$. Indeed, after discarding a bounded initial interval, \eqref{ineqder} holds on $\Omega_1$; its two right-hand terms have finite integral there by the preceding estimate, whereas the integral of the left-hand side equals the measure of $\Omega_1$.

If $m=1$, \eqref{eq:F1-rate} and the estimate for $\Omega_1$ prove both conclusions of the lemma. Henceforth, assume $m\ge2$.

Next, we show that, for every $i\ge2$,
\begin{equation} \label{eq:Fi-ratio}
    \liminf_{t\to \infty} \frac{\ln F_i(t)}{\ln F_{i-1}(t)}=1.
\end{equation}
 For sufficiently large $t,$ both $\ln F_i(t)$ and $\ln F_{i-1}(t)$ are negative. It is therefore more convenient to work with their negatives and to prove the equivalent statement 
 \[
\liminf_{t\to \infty} \frac{-\ln F_i(t)}{-\ln F_{i-1}(t)}=1.
\]
Since $\phi$ is convex and nonconstant, there exists $a>0$ such that $\phi'(y)\ge a$ for almost every sufficiently large $y$. Applying integration by parts to \eqref{Fm}, we obtain
 \[
F_{i-1}(t)\ge \frac{t^{i-1}}{i-1} e^{-\phi(tP)}+ \frac{a (P-R)}{i-1}\int_{tR}^{tP}  \frac{(y - tR)^{i-1}}{(P-R)^i} e^{-\phi (y)} \, dy.
\]
 Hence, \[
-\ln F_{i-1}(t) \le -\ln \frac{a (P-R)}{i-1} -\ln F_i(t).
\]
 It follows that \[
 \liminf_{t\to \infty} \frac{-\ln F_i(t)}{-\ln F_{i-1}(t)}\ge  \liminf_{t\to \infty} \frac{\ln \frac{a (P-R)}{i-1} -\ln F_{i-1}(t)}{-\ln F_{i-1}(t)}\ge 1,
\]
where the last inequality follows from $\ln F_{i-1}(t)\to-\infty$. This gives the lower bound in \eqref{eq:Fi-ratio}.

For the upper bound, use the derivative identity \eqref{derivin}.
 Assume, toward a contradiction, that there exists $\alpha >1$ such that 
 \[
 \liminf_{t\to \infty} \frac{-\ln F_i(t)}{-\ln F_{i-1}(t) } >\alpha.
\]
Equivalently, using \eqref{derivin}, we obtain
 \[
 \liminf_{t\to \infty} \frac{-\ln F_i(t)}{-\ln \left[ -\frac{P-R}{R(i-1)}F_i^{\prime}(t) +\frac{P t^{i-1}}{R(i-1)} e^{-\phi(tP)}\right]} >\alpha.
\]
 Then there exists $t_0>0$ such that for all $t>t_0,$ 
 \[
-\ln F_i^{\frac{1}{\alpha}}(t)>-\ln \left[ -\frac{P-R}{R(i-1)}F_i^{\prime}(t) +\frac{P t^{i-1}}{R(i-1)} e^{-\phi(tP)}\right],
\]
or
 \begin{equation} \label{unbd}
     1< -\frac{P-R}{R(i-1)} \frac{F_i^{\prime}(t)}{F_i^{\frac{1}{\alpha}}(t)} +\frac{P t^{i-1}}{R(i-1)} \frac{e^{-\phi(tP)}}{F_i^{\frac{1}{\alpha}}(t)}.
 \end{equation}
On the other hand, by monotonicity of $\phi$ we have
\[
F_i(t)\ge e^{-\phi(tP)}\int_{tR}^{tP}  \frac{(y - tR)^{i-1}}{(P-R)^i}\,dy = \frac{t^i}{i} e^{-\phi(tP)},
\] so
\begin{equation}\label{bound}
\frac{e^{-\phi(tP)}}{F_i^{\frac{1}{\alpha}}(t)}\le i^{\frac{1}{\alpha}} t^{\frac{-i}{\alpha}} e^{-(1-\frac{1}{\alpha}) \phi(tP)}.
\end{equation}
Combining \eqref{bound} and \eqref{unbd}, we obtain 
 \begin{equation}\label{ineq2}1< -\frac{P-R}{R(i-1)} \frac{F_i^{\prime}(t)}{F_i^{\frac{1}{\alpha}}(t)} +\frac{P}{R(i-1)} i^{\frac{1}{\alpha}} t^{i(1-\frac{1}{\alpha})-1} e^{-(1-\frac{1}{\alpha}) \phi(tP)}.\end{equation}
Since $1-1/\alpha>0$ and $\phi$ grows at least linearly, the second term on the right-hand side of \eqref{ineq2} is integrable. The first term has finite integral by eventual monotonicity. Integrating \eqref{ineq2} over $[t_0,\infty)$ gives a contradiction and establishes \eqref{eq:Fi-ratio} for every $i\ge2$.

As in \eqref{omega1}, the preceding argument yields a stronger property. Fix $\alpha>1$ and consider the set
 \begin{equation}\label{omega}
 \Omega_i=\left\{t\ge T:\frac{-\ln F_i(t)}{-\ln F_{i-1}(t)}>\alpha\right\}.
 \end{equation}
Then $|\Omega_i|<\infty$ for $i=2,3,\ldots$. Indeed, integrate \eqref{ineq2} over $\Omega_i$ after removing a bounded initial interval and use the preceding integrability estimate.

 To complete the proof of the lemma, we now combine \eqref{eq:F1-rate} and \eqref{eq:Fi-ratio}. The argument is adapted from \cite[Claim 2]{LZ}. For the sake of completeness, we reproduce it here.  Set 
\[
Y(t)=\frac{\ln F_1(t)}{\phi(Rt)}\quad \text{  and  } \quad X_i (t)=\frac{\ln F_i(t)}{ \ln F_{i-1}(t)}, \hspace{.5 cm} i = 2,\dots,\,m.
\]
By \eqref{eq:Fi-ratio} we obtain  $\liminf\limits_{t\to{\infty}} X_{i}(t)=1,$ for all $ i=2,\dots, m,$ and from \eqref{eq:F1-rate} we have
  $\limsup\limits_{t\to\infty}Y(t)= -1.$ 
 Define $X(t)=\prod_{i=2}^{m}X_{i}(t)$. Then the telescoping identity 
 $\ln{F_m}(t)=X(t) \ln{F_1(t)}$ implies \[
 \frac{\ln{F_m(t)}}{\phi (Rt)}=X(t)Y(t).
\]
 Hence, it suffices to prove
 $\limsup_{t\to\infty}X(t)Y(t) \ge -1.$
 Assume, to the contrary, that this is not the case. Then there exists $\alpha >1$ such that 
\[
\limsup_{t\to\infty}X(t)Y(t)<{-\alpha}<-1.
\]
Therefore, there exists $t_0>0$ such that
\begin{equation}\label{assumption} X(t)Y(t)<-\alpha \quad\text{for all } t>t_0.
\end{equation}
Using \eqref{eq:Fi-ratio} we may also assume that $X_i(t)>0$ for all $t>t_0.$ Next, consider the set  \[
A:=\left\{t>t_0:X(t)>\frac{\alpha+1}{2}\right\}.
\]

We claim that $|A|<\infty$. There are $m-1$ factors in $X(t)$. Thus, if
$X(t)>(\alpha+1)/2$, then for at least one $i\in\{2,\ldots,m\}$,
\[
X_i(t)>\left(\frac{\alpha+1}{2}\right)^{1/(m-1)}.
\]
Consequently,
\[
A\subset\bigcup_{i=2}^m
\left\{t:X_i(t)>\left(\frac{\alpha+1}{2}\right)^{1/(m-1)}\right\}.
\]
The threshold is strictly greater than $1$, so every set in this finite union has finite measure by \eqref{omega}. Hence $|A|<\infty$.
 Similarly, using that    $\frac{2\alpha}{\alpha+1}>1$  and \eqref{omega1} we also have 
\[
\left|\left\{t: Y(t) < -\frac{2 \alpha}{\alpha+1}\right\}\right|<\infty.
\]
Finally, since $X(t)>0$ for all sufficiently large $t$, we obtain
\begin{align}\label{measuretrick}
\Big|\Big\{t> t_0:   &  X(t)Y(t) < -\alpha \Big\}\Big| =\left|\left\{t> t_0: Y(t) < -\frac{\alpha}{X(t)}\right\}\right| \nonumber\\
&\le |A|+
\left|\left\{t> t_0: Y(t) < -\frac{\alpha}{X(t)} \text{  and } X(t)\le\frac{\alpha+1}{2}\right\}\right|\nonumber\\
&\le |A|+
\left|\left\{t: Y(t) < -\frac{2 \alpha}{\alpha+1}\right\}\right|<\infty,
\end{align}
which contradicts \eqref{assumption}. For the non-strict exceptional-set assertion, fix $\alpha>1$. For all sufficiently large $t$, $X(t)>0$, and
\[
\{t:X(t)Y(t)\le-\alpha\}
\subset
\left\{t:X(t)>\frac{\alpha+1}{2}\right\}
\cup
\left\{t:Y(t)\le-\frac{2\alpha}{\alpha+1}\right\}.
\]
The first set has finite measure by \eqref{omega} and the finite-product argument above, and the second has finite measure by \eqref{omega1}. Thus
\[
\left\{t>0:\frac{\ln F_m(t)}{\phi(Rt)}\le-\alpha\right\}
\]
has finite measure. Therefore
\[
\limsup_{t\to\infty}X(t)Y(t)\ge-1,
\qquad
\limsup_{t\to\infty}\frac{\ln F_m(t)}{\phi(Rt)}\ge-1.
\]
\end{proof}

\begin{proof}[Proof of Theorem~\ref{thm:first-mixed-ball}]
Set $R=r(K,L)$. A convex nonconstant nondecreasing function can be constant only on an initial interval. For $u\ge\phi(0)$, define its generalized inverse by
\[
\phi^{\leftarrow}(u)=\sup\{s\ge0:\phi(s)\le u\}.
\]
 For every $s\ge0$, the generalized inverse gives the layer-cake identity
\[
e^{-\phi(s)}
=\int_{\phi(s)}^\infty e^{-u}\,du
=\int_{\phi(0)}^\infty e^{-u}
\mathbf 1_{\{s\le\phi^{\leftarrow}(u)\}}\,du.
\]
By \eqref{eq:first-variation}, Tonelli's theorem, and this identity,
\begin{align}
\mu(tK;B_2^n)
&=\int_{\partial(tK)}e^{-\phi(\|y\|_L)}\,d\mathcal H^{n-1}(y)\nonumber\\
&=\int_{\phi(0)}^\infty e^{-u}\,
\mathcal H^{n-1}\!\left(\partial(tK)\cap\phi^{\leftarrow}(u)L\right)du.
\label{inform}
\end{align}
Since $RL\subset K$, the intersection in \eqref{inform} is empty whenever
$\phi^{\leftarrow}(u)<tR$. Consequently,
\begin{align}
\mu(tK;B_2^n)
&=\int_{\phi(tR)}^\infty e^{-u}\,
\mathcal H^{n-1}\!\left(\partial(tK)\cap\phi^{\leftarrow}(u)L\right)du\nonumber\\
&\le t^{n-1}\mathcal H^{n-1}(\partial K)e^{-\phi(tR)}.
\label{negativeln}
\end{align}
The elementary estimate \eqref{eq:lim} was established before its first use in Lemma~\ref{lemma:laplace-integrals}. It follows that
\[
\limsup_{t\to\infty}
\frac{\ln\mu(tK;B_2^n)}{\phi(tR)}\le-1.
\]
In addition, for every $\varepsilon>0$, the set on which
\[
\frac{\ln\mu(tK;B_2^n)}{\phi(tR)}\ge-1+\varepsilon
\]
is bounded.

We now prove the reverse estimate. Let $P>R$ and $c>0$ be the constants supplied by Lemma~\ref{lemma:geometric-cap}. For all sufficiently large $t$ and
$u\in[\phi(tR),\phi(tP)]$, one has
$\phi^{\leftarrow}(u)/t\in[R,P]$. Hence \eqref{inform} and
Lemma~\ref{lemma:geometric-cap} give
\begin{align*}
\mu(tK;B_2^n)
&\ge
t^{n-1}\int_{\phi(tR)}^{\phi(tP)}
e^{-u}\,
\mathcal H^{n-1}\!\left(
\partial K\cap\frac{\phi^{\leftarrow}(u)}tL
\right)du\\
&\ge
c\,t^{n-1}\int_{\phi(tR)}^{\phi(tP)}
e^{-u}
\left(\frac{\phi^{\leftarrow}(u)}t-R\right)^{n-1}du.
\end{align*}
For all sufficiently large $t$, the function $\phi$ is strictly increasing on
$[tR,tP]$. Since it is locally absolutely continuous and differentiable almost
everywhere, the one-dimensional change-of-variables formula yields
\[
\mu(tK;B_2^n)
\ge c\int_{tR}^{tP}
(x-tR)^{n-1}e^{-\phi(x)}\phi'(x)\,dx.
\]
There exist $a>0$ and $T>0$ such that $\phi'(x)\ge a$ for almost every
$x\ge T$. Therefore, for all sufficiently large $t$,
\begin{equation}\label{eq:lower-Fn}
\mu(tK;B_2^n)\ge c_1 F_n(t),
\end{equation}
where $F_n$ is defined in \eqref{Fm} and $c_1>0$ is independent of $t$.
Lemma~\ref{lemma:laplace-integrals} now gives
\[
\limsup_{t\to\infty}
\frac{\ln\mu(tK;B_2^n)}{\phi(tR)}\ge-1.
\]
Together with the upper estimate, this proves the first assertion.

Finally, fix $\varepsilon>0$. The upper-tail set is bounded by the estimate
following \eqref{eq:lim}. Moreover, \eqref{eq:lower-Fn} gives, for all
sufficiently large $t$,
\[
\frac{\ln\mu(tK;B_2^n)}{\phi(Rt)}
\ge
\frac{\ln F_n(t)}{\phi(Rt)}
-\frac{C}{\phi(Rt)}.
\]
Thus, up to a bounded set, the lower-tail set on which the left-hand side is at
most $-1-\varepsilon$ is contained in
\[
\left\{t:
\frac{\ln F_n(t)}{\phi(Rt)}
\le-1-\frac{\varepsilon}{2}\right\},
\]
which has finite measure by Lemma~\ref{lemma:laplace-integrals}. This proves
the exceptional-set assertion.
\end{proof}
We now extend the geometric large-deviation-type principle established in Theorem~\ref{thm:first-mixed-ball} for $B_2^n$ to an arbitrary convex body $M$ containing the origin in its interior. The key observation is that the support function $h_M:\mathbb S^{n-1}\to\mathbb R$ is bounded above and below by positive constants. This allows us to compare $\mu(tK;M)$ with $\mu(tK;B_2^n)$ and thereby transfer the asymptotic behavior obtained in the reference case.

\begin{theorem} \label{GLDP1}
Let $K,M\subset\mathbb R^n$ be convex bodies with $0\in\operatorname{int}K\cap\operatorname{int}M$. Then
\[
\limsup_{t\to\infty}\frac{\ln\mu(tK;M)}{\phi(r(K,L)t)}=-1.
\]
Moreover, for every $\varepsilon>0$, the set
\[
\left\{t>0:
\left|\frac{\ln\mu(tK;M)}{\phi(r(K,L)t)}+1\right|\ge\varepsilon\right\}
\]
has finite Lebesgue measure.
\end{theorem}

\begin{proof}
      By the integral representation of mixed measures \eqref{eq:first-variation}, we have
\begin{align*}
\mu(tK;M)
&=\int_{\partial(tK)}h_M(n_{tK}(y))e^{-\phi(\|y\|_L)}\,d\mathcal H^{n-1}(y)\\
&=t^{n-1}\int_{\partial K}h_M(n_K(y))e^{-\phi(t\|y\|_L)}\,d\mathcal H^{n-1}(y).
\end{align*}
Since $M$ is a convex body with $0\in\operatorname{int}M$, its support function $h_M$ is strictly positive and continuous on $\mathbb S^{n-1}$. Hence there exist constants $0<\alpha\le\beta<\infty$ such that $h_M(\theta)\in[\alpha,\beta]$ for every $\theta\in\mathbb S^{n-1}$.
    Consequently, 
    \begin{equation}\label{eq:mixed-comparison}
        \alpha \mu (tK; B_2^n) \le \mu(tK;M) \le \beta \mu (tK; B_2^n).
    \end{equation}
 Taking logarithms, dividing by $\phi(r(K,L)t)$, and taking the $\limsup$ as $t\to\infty$, the two-sided comparisonm together with Theorem~\ref{thm:first-mixed-ball} yields 
      \[
  \limsup _{t\to \infty}\frac{\ln \mu (tK; M) }{\phi( r(K, L) t)} = -1, 
\]
    which proves the first assertion.
    
We now proceed to prove the second part of the theorem. Fix $\varepsilon>0$ and define
\[
\Omega_M(\varepsilon)=\left\{t>0:
\left|\frac{\ln\mu(tK;M)}{\phi(r(K,L)t)}+1\right|\ge\varepsilon\right\}.
\]
    By \eqref{eq:mixed-comparison}, we obtain
    \begin{align*}
     \left( \frac{\ln \mu (tK; B_2^n) }{\phi( r(K,L)t)}+1\right)+\frac{\ln \alpha}{\phi( r(K,L)t)}\le\frac{\ln \mu(tK;M) }{\phi( r(K,L)t)}+1&\\\le\left( \frac{\ln \mu (tK; B_2^n) }{\phi( r(K,L)t)}+1\right)+\frac{\ln \beta}{\phi( r(K,L)t)}.
    \end{align*}
Hence,
\[
\left|\left(\frac{\ln\mu(tK;B_2^n)}{\phi(r(K,L)t)}+1\right)
-\left(\frac{\ln\mu(tK;M)}{\phi(r(K,L)t)}+1\right)\right|
\le\frac{\max\{|\ln\alpha|,|\ln\beta|\}}{\phi(r(K,L)t)}.
\]
Since $\phi(r(K,L)t)\to\infty$, there exists $t_0>0$ such that
\[
\frac{\max\{|\ln\alpha|,|\ln\beta|\}}{\phi(r(K,L)t)}
\le\frac{\varepsilon}{2}
\qquad\text{for all }t\ge t_0.
\]
For such $t$, the reverse triangle inequality implies 
\[
\Big| \frac{\ln \mu (tK; M) }{\phi( r(K,L)t)}+1\Big|\ge \varepsilon \quad \Longrightarrow\quad \Big|\frac{\ln \mu(tK;B_2^n) }{\phi( r(K,L)t)}+1\Big|\ge \frac{\varepsilon}{2}.
\]
Equivalently,
\[
\Omega_M(\varepsilon)\cap[t_0,\infty)
\subset \Omega_{B_2^n}\!\left(\frac{\varepsilon}{2}\right),
\]
where, for any $\delta>0$,
\[
\Omega_{B_2^n}(\delta)=\left\{t>0:
\left|\frac{\ln\mu(tK;B_2^n)}{\phi(r(K,L)t)}+1\right|\ge\delta\right\}.
\]
By Theorem~\ref{thm:first-mixed-ball}, the set
$\Omega_{B_2^n}(\varepsilon/2)$ has finite Lebesgue measure. Hence
$\Omega_M(\varepsilon)\cap[t_0,\infty)$ has finite measure. The remaining
portion of $\Omega_M(\varepsilon)$ is contained in a bounded interval and
therefore also has finite measure. Thus $|\Omega_M(\varepsilon)|<\infty$, as
claimed.
\end{proof}

\begin{remark}\label{rem:nonlimit}
In general, the $\limsup$ in Theorem~\ref{GLDP1} cannot be replaced by a
limit. To see this, consider a planar case with $ M=B_2^2$ and 
\[
\begin{aligned}
K&=[-1,1]^2,\\
L&=\operatorname{conv}\{\pm e_1,\pm e_2\}
 =\{x\in\mathbb R^2:|x_1|+|x_2|\le1\}.
\end{aligned}
\]
Then $r(K,L)=1$ and $\|x\|_L=|x_1|+|x_2|$. 

\begin{figure}[H]
\centering
\begin{tikzpicture}[scale=2.2, >=Stealth]

    \foreach \r in {0.3, 0.6, 1.3, 1.6, 1.9} {
        \draw[green!60!black!25, thin] (0,\r) -- (\r,0) -- (0,-\r) -- (-\r,0) -- cycle;
    }

    \fill[cyan!25] (-1,-1) rectangle (1,1);
    \draw[line width=1.2pt, draw=black] (-1,-1) rectangle (1,1);

    \fill[green!45!black!35] (0,1) -- (1,0) -- (0,-1) -- (-1,0) -- cycle;
    \draw[line width=1.2pt, draw=black] (0,1) -- (1,0) -- (0,-1) -- (-1,0) -- cycle;

\draw[dash dot dot, gray!80!black, thick] (0,0) circle (1);
    
    \draw[->, line width=0.8pt] (-2.2,0) -- (2.3,0) node[right] {$x_1$};
    \draw[->, line width=0.8pt] (0,-1.6) -- (0,1.8) node[above] {$x_2$};
    \node[below left] at (0,0) {$0$};

    \node[above left] at (-1,1) {$(-1,1)$};
    \node[above right] at (1,1) {$(1,1)$};
    \node[below left] at (-1,-1) {$(-1,-1)$};
    \node[below right] at (1,-1) {$(1,-1)$};
    \node[above right] at (-0.95, 0.65) {$K$};

    \node[above right] at (0,1) {$(0,1)$};
    \node[below right] at (0,-1) {$(0,-1)$};
    \node[above right] at (1,0) {$(1,0)$};
    \node[above left] at (-1,0) {$(-1,0)$};
    \node[above, align=center] at (0.4, 0.19) {$L $};

\end{tikzpicture}
\caption{Geometric configuration for Remark~\ref{rem:nonlimit}}.

\label{fig:remark31}
\end{figure}

Let $\mu$ have density
$c_\phi e^{-\phi(\|x\|_L)}$, where $c_\phi>0$ is the normalization constant.
Since $h_{B_2^2}\equiv1$ on $\mathbb S^1$, the integral representation
\eqref{eq:first-variation} gives
\[
\mu(tK;B_2^2)
=c_\phi\int_{\partial(tK)}
e^{-\phi(|x_1|+|x_2|)}\,d\mathcal H^1(x).
\]
The four sides of $tK$ give equal contributions. Parametrizing the right
vertical side by $x=(t,s)$, where $-t\le s\le t$, we obtain
\[
c_\phi\int_{-t}^{t}e^{-\phi(t+|s|)}\,ds
=2c_\phi\int_0^t e^{-\phi(t+s)}\,ds
=2c_\phi\int_t^{2t}e^{-\phi(u)}\,du.
\]
Therefore,
\[
\mu(tK;B_2^2)
=8c_\phi\int_t^{2t}e^{-\phi(s)}\,ds.
\]
For the convex nondecreasing function $\phi$ constructed in
\cite[Remark~2]{LZ}, there exists a sequence $t_k\to\infty$ such that
$\phi'(t_k)>e^{\phi(t_k)}$. By convexity,
\[
\int_{t_k}^{2t_k}e^{-\phi(s)}\,ds
\le \int_{t_k}^{\infty}e^{-\phi(s)}\,ds
\le \frac{e^{-\phi(t_k)}}{\phi'(t_k)}
<e^{-2\phi(t_k)}.
\]
Consequently,
\[
\liminf_{t\to\infty}
\frac{\ln\mu(tK;B_2^2)}{\phi(t)}
\le -2,
\]
whereas Theorem~\ref{GLDP1} gives that the corresponding $\limsup$ equals
$-1$. Thus, the normalized expression in Theorem~\ref{GLDP1} need not
converge.
\end{remark}

\begin{remark}
  The mixed measure is $1$-homogeneous in its second argument, so $\mu(tK;tM)=t\mu(tK;M)$. Therefore, the large-deviation-type asymptotics obtained above for $\mu(tK;M)$ also apply to $\mu(tK;tM)$.
\end{remark}

\section{Applications}\label{comparison}
\subsection{Large-Deviation-Type Principle for Complements}

We first use Theorem~\ref{GLDP1} to recover the large-deviation-type principle proved in \cite{LZ}.

\begin{lemma}\label{lemmaLZ}
For a convex body $K\subset\mathbb R^n$ with $0\in\operatorname{int}K$,
\begin{equation}\label{ldL}
\limsup_{t\to\infty}\frac{\ln{\mu((tK)^c)}}{\phi(r(K,L)t)}=-1.
    \end{equation}
\end{lemma}
\begin{proof} 
It was observed in \cite{LG} that, directly from the definition \eqref{eq:mixed-measure-definition}, one has  $\mu(xK; K)=\frac{d}{dx}\mu(xK)$, for all $x>0$. Thus 
\begin{equation}\label{form}
\frac{\ln\mu((tK)^c)}{\phi(r(K,L)t)} = \frac{\ln (1- \mu (tK))}{\phi( r(K, L) t)}
= \frac{\ln (\int_t^\infty \mu(xK; K) dx)}{\phi( r(K, L) t)}.
\end{equation}
Using Theorem~\ref{GLDP1}, we obtain that for every $\delta\in(0,1/2)$ there exists $x_\delta>0$ such that for all $x>x_\delta$
\[
\mu(xK; K) \le e^{-(1-\delta) \phi( r(K, L) x)}.
\]
Since $\phi$ is convex and nonconstant, there exists $a>0$ such that $\phi'(y)\ge a$ for almost every sufficiently large $y$. Consequently, for all sufficiently large $t$ and every $x\ge t$,
\[
{\phi(r(K,L)x)\ge\phi(r(K,L)t)+c(x-t).}
\]
Therefore,
\[
\int_t^\infty \mu(xK; K) dx \le c' e^{-(1-\delta)\phi( r(K, L) t)}.
\]
Together with \eqref{form}, this gives 
\[
\limsup _{t\to \infty}\frac{\ln\mu((tK)^c)}{\phi(r(K,L)t)} \le -1+\delta.
\]
Letting $\delta\downarrow0$ proves the upper bound.

For the lower bound, set
\[
F(t)=\int_t^\infty\mu(xK;K)\,dx
\]
and, for $\varepsilon>0$, define
\begin{equation}\label{con}
A_\varepsilon=\left\{t\ge0:\frac{\ln F(t)}{\ln\mu(tK;K)}\ge1+\varepsilon\right\}.
\end{equation}
We claim that $|A_\varepsilon|<\infty$. Since $F'(t)=-\mu(tK;K)$, for all sufficiently large $t\in A_\varepsilon$, both logarithms are negative and the defining inequality is equivalent to
\[
1\le\frac{\mu(tK;K)}{F(t)^{1/(1+\varepsilon)}}
=-\frac{F'(t)}{F(t)^{1/(1+\varepsilon)}}.
\]
The right-hand side is nonnegative and integrable on every tail:
\[
\int_T^\infty-\frac{F'(t)}{F(t)^{1/(1+\varepsilon)}}\,dt
=\frac{1+\varepsilon}{\varepsilon}F(T)^{\varepsilon/(1+\varepsilon)}<\infty.
\]
Hence $A_\varepsilon$ has finite measure.

Now we are ready to prove the lower bound. Let 
\[
X(t)= \frac{\ln F(t)}{\ln \mu(tK;K)} \mbox{ and } Y(t)= \frac{\ln \mu (tK; K) }{\phi( r(K, L) t)}.
\]
We need to show that 
$\limsup\limits_{t\to \infty} X(t)Y(t) \ge -1.$
Assume, toward a contradiction, that there exist $\alpha>1$ and $t_\alpha>0$ such that
\[
X(t)Y(t)<-\alpha\qquad(t>t_\alpha).
\]
For all sufficiently large $t$, both logarithms defining $X(t)$ are negative, so $X(t)>0$. Hence
\[
\begin{aligned}
\{t>t_\alpha:X(t)Y(t)<-\alpha\}
&\subset
\left\{t>t_\alpha:X(t)\ge\frac{\alpha+1}{2}\right\}\\
&\quad\cup
\left\{t>t_\alpha:Y(t)<-\frac{2\alpha}{\alpha+1}\right\}.
\end{aligned}
\]
The first set has finite measure by the preceding estimate for $A_\varepsilon$, with $\varepsilon=(\alpha-1)/2$. The second has finite measure by the exceptional-set assertion of Theorem~\ref{GLDP1}, since $2\alpha/(\alpha+1)>1$. This contradicts the fact that the set on the left contains a full tail interval and completes the proof of the lower bound.
\end{proof}

\subsection{A Comparison Principle}
We use Theorem \ref{GLDP1} to derive a comparison principle for convex bodies based on the large-deviation-type asymptotics of mixed measures. The result shows that asymptotic dominance of mixed measures under dilation forces a geometric inclusion between the underlying convex bodies.

We continue with the normalized log-concave measure $\mu$ defined at the beginning of Section~\ref{twobodies}.
\begin{theorem}\label{application}
Let $K, M\subset \mathbb{R}^n$ be convex bodies containing the origin in their interiors. Suppose that there exists $R>0$ such that
\begin{equation}\label{assrt}
\mu(tRL; M) \ge \mu (tK; M) \qquad\text{for all sufficiently large } t. 
\end{equation}
Then $RL\subseteq K$.
\end{theorem}

\begin{proof} Assume, toward a contradiction, that \eqref{assrt} holds but $RL\not\subseteq K$. Then the maximal dilate of $L$ contained in $K$ has radius $rR$ for some $r\in(0,1)$. It follows from \eqref{assrt} that
 \begin{equation*}
      \frac{\ln{\mu(tRL; M)}}{\phi(tR)}\ge\frac{\ln{\mu (tK; M)}}{\phi(trR)}\frac{\phi(trR)}{\phi(tR)}.
 \end{equation*}
By the convexity of $\phi$, we obtain
 \[
\phi(trR)=\phi(r(tR)+(1-r)0)\le r\phi(tR)+(1-r)\phi(0),
\] 
 or 
 \[
\frac{\phi(trR)}{\phi(tR)}\le r +(1-r)\frac{\phi(0)}{\phi(tR)}.
\]
 Since $\phi(tR) \to \infty $ as $t\to\infty,$ there exist constants $r'\in (0,1)$ and $t_0>0$ such that \[
\frac{\phi(trR)}{\phi(tR)}\le r'\quad\text{for all } t>t_0.
\] Moreover, for large $t,$ the quantities $\ln{\mu(tRL; M)}$ and $\ln{\mu(tK; M)}$ are negative by \eqref{eq:mixed-comparison} and \eqref{negativeln}. Hence, for all $t>t_0$
  \begin{equation}\label{comparison-ineq}
      \frac{\ln{\mu(tRL; M)}}{\phi(tR)}\ge r'\;\frac{\ln{\mu (tK; M)}}{\phi(trR)}.
 \end{equation} 
Taking the $\limsup$ as $t\to\infty$ on both sides of \eqref{comparison-ineq} and applying Theorem~\ref{GLDP1}, we obtain $-1\ge-r'$, which contradicts $r'<1$. Hence $RL\subseteq K$.
\end{proof}
\begin{remark} Saracco and Stefani \cite{SS} proved that, for a Borel measure $\nu$ on $\mathbb R^n$ with locally integrable density, if $\nu^+(\partial K)\le\nu^+(\partial L)$ for all convex sets $K\subseteq L$, then $\nu$ is a multiple of Lebesgue measure. A related phenomenon is discussed in \cite{WBMTII}: if $\nu$ is a Radon measure on $\mathbb R^n$ such that
$\nu^+(\partial(K+Q))\ge\nu^+(\partial K)$
for every convex body $K$ and compact convex set $Q$, then $\nu$ is a multiple of Lebesgue measure. For a log-concave measure of the form $d\mu(x)=e^{-\phi(\|x\|_L)}\,dx$, Theorem~\ref{application}, applied with $M=B_2^n$, shows that if $\mu^+(\partial(tL))\ge\mu^+(\partial(tK))$ for all sufficiently large $t$, then $L\subseteq K$.
\end{remark}

\newpage
\section{Three-Body Mixed Measures in the Plane}\label{three}
We investigate the large-deviation behavior of mixed measures involving three convex bodies in $\mathbb R^2$. We begin with the Gaussian measure, for which the principle can be verified explicitly, and then consider general log-concave measures of the form $d\mu(x)=e^{-\phi(\|x\|_L)}\,dx$.

\subsection{{The Gaussian Case}} We first treat the planar Gaussian measure. Although the result is a special case of Theorem~\ref{theorem:ldpmix}, a separate proof highlights the geometric features governing the asymptotic behavior and clarifies the roles of curvature and support functions.
\begin{theorem} \label{thm:gaussian-rate}
Let $A,B,C\subset\mathbb R^2$ be convex bodies with the origin in their interiors, and assume that $A\in C_+^2$. Then
    \[
    \lim_{t \to \infty} \frac{\ln \left[ -\gamma_2(tA; B, C)\right]}{t^2/2}  = -   r(A, B_2^2)^2.
\]
\end{theorem}
\begin{proof}
Our goal is to show that
\begin{equation}\label{support}
    \lim_{t \to \infty} \frac{\ln \left[ -\gamma_2(tA; B, C)\right]}{t^2/2}  = -  \min\limits_{\theta \in [0, 2\pi]} \left[  h_A^2(\theta) + (h_A'(\theta))^2 \right].\end{equation}
    Once this limit is established, the statement of the theorem follows immediately. Indeed, $x_A(\theta)=h_A(\theta)u(\theta)+h_A'(\theta)u'(\theta)$ parametrizes $\partial A$, so
$|x_A(\theta)|^2=h_A(\theta)^2+h_A'(\theta)^2$. The minimum of $|x|$ over $x\in\partial A$ is the radius $r(A,B_2^2)$ of the largest centered Euclidean ball contained in $A$. Hence the minimum in \eqref{support} equals $r(A,B_2^2)^2$. 

We now analyze \eqref{eq:gaussian-planar}. For $t>0$ and almost every
$\theta\in[0,2\pi]$, set
\[
G_t(\theta)=
-\Big[h_B(\theta)h_C(\theta)
\big(1-t^2h_A(\theta)f_A(\theta)\big)
-h'_B(\theta)h'_C(\theta)\Big].
\]
The functions $h'_B$ and $h'_C$ are essentially bounded on $[0,2\pi]$; that is, each is bounded outside a set of Lebesgue measure zero.
Thus there exist positive constants $c_1,c_2,c_3$, independent of $\theta$,
such that, for every $t>c_3$ and almost every $\theta\in[0,2\pi]$,
\[
\begin{aligned}
c_1t^2h_B(\theta)h_C(\theta)h_A(\theta)f_A(\theta)
&\le G_t(\theta)\\
&\le c_2t^2h_B(\theta)h_C(\theta)h_A(\theta)f_A(\theta).
\end{aligned}
\]
Since $h_Bh_Ch_Af_A$ is continuous and strictly positive on $[0,2\pi]$,
there exist constants $c',c''>0$ such that, for every $t>c_3$,
\[
c't^2\le G_t(\theta)\le c''t^2
\]
for almost every $\theta\in[0,2\pi]$.
Set $\psi(\theta)=\frac12\big(h_A^2(\theta)+(h_A'(\theta))^2\big)$.
 Formula~\eqref{eq:gaussian-planar} and the definition of $G_t$ now give the explicit sandwich
\[
\frac{c't^2}{2\pi}\int_0^{2\pi}e^{-t^2\psi(\theta)}\,d\theta
\le-\gamma_2(tA;B,C)
\le\frac{c''t^2}{2\pi}\int_0^{2\pi}e^{-t^2\psi(\theta)}\,d\theta.
\]
In particular, $-\gamma_2(tA;B,C)>0$ for all sufficiently large $t$, and the polynomial prefactor does not affect the logarithmic rate.

To finish the proof, it is enough to show that
    \[
\lim_{t \to \infty} \frac{\ln \left[ \int_0^{2\pi}e^{-t^2 \,\psi(\theta)}\,d\theta\right]}{t^2}  = - \min\limits_{\theta \in [0, 2\pi]} \psi(\theta).
\]
This limit follows from Laplace's principle. We include the argument for completeness. The upper bound is immediate:
 \[
 \frac{\ln \left[ \int_0^{2\pi}e^{-t^2\,\psi(\theta)}\,d\theta\right]}{t^2}  \le  - \min\limits_{\theta \in [0, 2\pi]} \psi(\theta) +\frac{\ln(2\pi)}{t^2}. 
\]
 To establish the matching lower bound, let $r=  \min\limits_{\theta \in [0, 2\pi]} \psi(\theta)$. Suppose, toward a contradiction, that there exist $\varepsilon >0 $ and a sequence $t_k \to \infty$ such that 
$\ln \int_0^{2\pi}e^{-t_k^2\,\psi(\theta)}\,d\theta  <  - (r + \varepsilon) t_k^2$ for all $k \in {\mathbb N}$.
Thus
\begin{equation}\label{inq1} 
\int_0^{2\pi}e^{-t_k^2\,[\psi(\theta)- (r +\varepsilon)]} \,d\theta< 1.
\end{equation}
For $\delta >0,$ define the level set $A_\delta =\{\theta\in [0, 2\pi]:  \psi(\theta) \le r+\delta\}$. By continuity of $\psi,$ this set has positive Lebesgue measure $|A_\delta|>0.$ Pick $\delta=\varepsilon/2$ and note that
 for $\theta \in A_{\varepsilon/2}$ one has
$
\psi(\theta)-(r+\varepsilon)\le-\varepsilon/2,
$
and therefore, using \eqref{inq1}, we obtain
\[
 1>  \int_{A_{\varepsilon/2}} e^{-t_k^2 \,[\psi(\theta)- (r +\varepsilon)]} \,d\theta \ge   \int_{A_{\varepsilon/2}} e^{t_k^2 \,\varepsilon/2} \,d\theta  = e^{t_k^2 \,\varepsilon/2}|A_{\varepsilon/2}|.
\]
Finally, this yields $e^{-t_k^2\varepsilon/2}>|A_{\varepsilon/2}|$. Letting $t_k\to\infty$ gives $|A_{\varepsilon/2}|=0$, a contradiction. This establishes \eqref{support} and completes the proof.

\end{proof}

 \subsection{General Gauge-Dependent Mixed Measures in the Plane}

 The following theorem provides an explicit integral representation of the weighted mixed measure $\mu(A;B,C)$ in the plane. It extends the Gaussian formula \eqref{eq:gaussian-planar} to a broader class of log-concave measures with densities of the form $e^{-\phi(\|x\|_L)}$. We use the angular notation introduced before formula~\eqref{eq:gaussian-planar}.

\begin{theorem}\label{theorem_mu}
Let $A,B,C,L\subset\mathbb R^2$ be convex bodies with the origin in their interiors, assume $A,L\in C_+^2$, and let  $\phi:[0,\infty)\to[0,\infty)$ be nondecreasing and convex, with $\phi\in C^2((0,\infty))$. Assume that $\mu$ is a probability measure with density $e^{-\phi(\|x\|_L)}$. Then

\begin{align*}
\mu(A;B,C)
&=\int_0^{2\pi} e^{-\phi(\|x_A(\theta)\|_L)}
\Big[h_B(\theta)h_C(\theta)\\
&\qquad-h_B(\theta)h_C(\theta)\phi'(\|x_A(\theta)\|_L)f_A(\theta)
\big\langle\nabla\|x_A(\theta)\|_L,u(\theta)\big\rangle\\
&\qquad-h'_B(\theta)h'_C(\theta)\Big] \,d\theta.
\end{align*}
Here $h'_B$ and $h'_C$ are understood in the almost-everywhere sense; changing their values on a null set does not affect the formula.
 \end{theorem}

\begin{proof} Our goal is to simplify \eqref{mimeas3} and \eqref{defmeas} in $\mathbb R^2$ for $d\mu(x)=e^{-\phi(\|x\|_L)}\,dx$. We first assume that $B,C\in C_+^2$ and remove this additional assumption at the end by approximation. We begin by expressing $\nabla h_A(u)$ in the standard orthonormal frame $\{u(\theta),u'(\theta)\}$.
\begin{equation} \label{h} 
    \nabla h_A(u( \theta))=h_A(\theta)u(\theta) + h'_A(\theta)u'(\theta).
\end{equation}
Differentiating $x_A(\theta)=\nabla h_A(u(\theta))$ and using $u''(\theta)=-u(\theta),$ we obtain
\begin{equation}\label{curvature2}
x'_A(\theta)=\big(h''_A(\theta) +h_A(\theta)\big)\,u'(\theta)=f_A(\theta)\,u'(\theta),
\end{equation}
where we used the standard formula for the curvature function in ${\mathbb R}^2$ (see, for example, \cite{Ga}). Using \eqref{defmeas}, we observe that  $\mu (A; B, C)$ can be written as 
\begin{align} \label{eq:three-body-decomposition}
    \mu (A; B, C)&= \int _{\mathbb S^1}  h_C(u)e^{-\phi(\| n_A^{-1} (u) \|_{L})}\, dS_{B}(u) \, \nonumber\\
    &+ \int _{\mathbb S^1} h_C(u) \big\langle \nabla  e^{-\phi(\| n_A^{-1} (u) \|_{L})},\, \nabla h_B(u) \big\rangle \, dS_A(u).
\end{align}
Since $A\in C^2_+$, the inverse Gauss map satisfies $n_A^{-1}(u(\theta))=\nabla h_A(u(\theta))=x_A(\theta)$. Moreover, the surface area measures satisfy $dS_A(\theta)= f_A(\theta)\, d\theta$. Since $B\in C^2_+$, we have  $dS_B (\theta)=f_B(\theta)d\theta=(h''_B(\theta) +h_B(\theta))\, d\theta.$ Substituting these observations into \eqref{eq:three-body-decomposition} and changing variables from $u$ to $\theta,$ we obtain
\begin{align}\label{mix3}
    \mu (A; B, C) & = \int _{0} ^ {2\pi}  h_C(\theta)e^{-\phi(\| x_A(\theta) \|_{L})}\, (h''_B(\theta) +h_B(\theta)) \,d\theta  \nonumber\\ &  +\int _{0} ^ {2\pi} h_C(\theta) \big\langle \nabla e^{-\phi(\|x_A(\theta) \|_{L})}, h_B(\theta)u(\theta) + h'_B(\theta)u'(\theta) \big\rangle \,  f_A(\theta) \, d\theta. 
\end{align}
We also observe that 
\[
 \nabla e^{-\phi(\|x_A(\theta)\|_{L})}=-e^{-\phi(\| x_A(\theta)\|_{L})}\,\phi'(\| x_A(\theta)\|_{L})\,\nabla \|x_A(\theta) \|_{L},
\]
and substituting this equality into \eqref{mix3} yields
\begin{align}\label{eq:integration-by-parts-start}
    \mu (A; B, C)  & = \int _{0} ^ {2\pi}  h_C(\theta)\,e^{-\phi(\| x_A(\theta) \|_{L})}\, (h''_B(\theta) +h_B(\theta)) \,d\theta \nonumber \\ & \quad -\int _{0} ^ {2\pi}  h_C(\theta)\,e^{-\phi(\| x_A(\theta) \|_{L})}\,\phi'(\| x_A(\theta) \|_{L}) \, \nonumber\\ \times&\Big[h_B(\theta) \big\langle \nabla \|x_A(\theta) \|_{L}\,,u(\theta)\big\rangle +h'_B(\theta)\big\langle \nabla \|x_A(\theta) \|_{L},\,u'(\theta) \big\rangle \Big]\,f_A(\theta) \, d\theta.
\end{align}
Observe that 
\begin{align}\label{ident}
d\big[e^{-\phi(\| x_A(\theta) \|_{L})} \big]&=-e^{-\phi(\| x_A(\theta) \|_{L})}\,\phi'(\| x_A(\theta) \|_{L})\,\langle \nabla\| x_A(\theta) \|_{L},x'_A(\theta) \rangle \, d\theta \nonumber\\
&=-e^{-\phi(\| x_A(\theta) \|_{L})}\,\phi'(\| x_A(\theta) \|_{L})\,f_A(\theta)\,\langle \nabla\| x_A(\theta) \|_{L},u'(\theta) \rangle \, d\theta,
\end{align}
where we used \eqref{curvature2}.

Using \eqref{ident}, we integrate by parts the terms involving $\langle\nabla\|x_A(\theta)\|_L,u'(\theta)\rangle$ in \eqref{eq:integration-by-parts-start}; the resulting term cancels the second-order derivatives of $h_B$. Collecting the remaining terms gives the stated formula when $B,C\in C_+^2$.

For arbitrary convex bodies $B$ and $C$, choose $B_j,C_j\in C_+^2$ so that $h_{B_j}\to h_B$ and $h_{C_j}\to h_C$ uniformly, while $h'_{B_j}\to h'_B$ and $h'_{C_j}\to h'_C$ in $L^2(\mathbb S^1)$. Such sequences may be obtained by smoothing the support functions on $\mathbb S^1$ and adding a vanishing positive constant. The terms containing $h_{B_j}h_{C_j}$ then converge uniformly, and
\[
\|h'_{B_j}h'_{C_j}-h'_Bh'_C\|_{L^1}
\le
\|h'_{B_j}-h'_B\|_{L^2}\|h'_{C_j}\|_{L^2}
+\|h'_B\|_{L^2}\|h'_{C_j}-h'_C\|_{L^2}
\longrightarrow 0.
\]
The continuity under this approximation established in the proof of \cite[Theorem~2.7]{WBMTI} also gives $\mu(A;B_j,C_j)\to\mu(A;B,C)$. Passing to the limit proves the formula for arbitrary $B$ and $C$ and completes the proof.
\end{proof}
We now establish a large-deviation-type principle for mixed measures involving three planar convex bodies. This result extends the Gaussian case treated in Theorem~\ref{thm:gaussian-rate} to a broad class of log-concave measures whose densities depend on a gauge.
\begin{theorem}\label{theorem:ldpmix}
Let $A,B,C,L\subset\mathbb R^2$ be convex bodies with the origin in their interiors, and suppose that $A,L\in C_+^2$.
 Let $\phi:[0,\infty)\to[0,\infty)$ be nondecreasing, nonconstant, and convex, with $\phi\in C^2((0,\infty))$. Assume that $\mu$ is a probability measure with density $e^{-\phi(\|x\|_L)}$.  Suppose that
\begin{equation}\label{restr}\lim_{t \to \infty}\frac{\ln \phi'(t)}{\phi(t)}=0.
\end{equation}
 Here $\phi'(t)>0$ for all sufficiently large $t$, so the logarithm in \eqref{restr} is well defined on a tail interval.
Then
    \[
    \lim_{t\to\infty}\frac{\ln[-\mu(tA;B,C)]}{\phi(r(A,L)t)}=-1.
\]
\end{theorem}

\begin{proof}
Throughout the proof, $c_1, c_2, \ldots$ denote strictly positive constants independent of $t,$ though they may depend on $\phi$ and on the convex bodies $A, B, C$ and $L$. We first obtain an absolute upper estimate without taking logarithms. The lower-bound argument will then prove that $-\mu(tA;B,C)$ is positive for all sufficiently large $t$.

We note that the convex bodies $B$ and $C$ need not be smooth. Indeed, their support functions are Lipschitz on $\mathbb S^1$, so $h'_B$ and $h'_C$ exist almost everywhere and are essentially bounded. Thus all formulas below are understood almost everywhere, as in Theorem~\ref{theorem_mu}.

We start from the planar integral representation in Theorem~\ref{theorem_mu}. Using the scaling relations $x_{tA}(\theta)=t\,x_A(\theta)$, $f_{tA}(\theta)=t\,f_A(\theta)$, and $\nabla\|x_{tA}(\theta)\|_L=\nabla\|x_A(\theta)\|_L$, we obtain
\begin{align*}
\mu(tA;B,C)
&=\int_0^{2\pi}e^{-\phi(t\|x_A(\theta)\|_L)}\\
&\quad\times\Big\{
h_B(\theta)h_C(\theta)
\Big[1-t\phi'(t\|x_A(\theta)\|_L)f_A(\theta)\\
&\hspace{2.2cm}\times
\big\langle\nabla\|x_A(\theta)\|_L,u(\theta)\big\rangle\Big]
-h'_B(\theta)h'_C(\theta)
\Big\}\,d\theta.
\end{align*}
We first record an absolute upper estimate. Since $\phi'\ge0$,
\begin{align*}
|\mu(tA;B,C)|
&\le\int_0^{2\pi}e^{-\phi(t\|x_A(\theta)\|_L)}
\Big\{|h_B(\theta)h_C(\theta)|+|h'_B(\theta)h'_C(\theta)|\\
&\qquad+t\phi'(t\|x_A(\theta)\|_L)
\big|h_B(\theta)h_C(\theta)f_A(\theta)\big|\\
&\hspace{4.1cm}\times
\big|\langle\nabla\|x_A(\theta)\|_L,u(\theta)\rangle\big|
\Big\}\,d\theta.
\end{align*}
The functions $h_B,h_C,f_A$, and $|\nabla\|x_A\|_L|$ are bounded on $[0,2\pi]$, while $h'_B$ and $h'_C$ are essentially bounded. Also, $\phi'(t\|x_A(\theta)\|_L)\ge c_2>0$ uniformly in $\theta$ for all sufficiently large $t$. Hence
\begin{equation}\label{upperbound}
|\mu(tA;B,C)|\le
c_1t\int_0^{2\pi}e^{-\phi(t\|x_A(\theta)\|_L)}
\phi'(t\|x_A(\theta)\|_L)\,d\theta.
\end{equation}
By \eqref{restr}, for every $\varepsilon>0$ there exists $x_0>0$ such that
\[
\phi'(x)\le e^{\varepsilon\phi(x)}
\qquad(x>x_0).
\]
Since $\|x_A(\theta)\|_L$ is continuous and positive, this estimate holds uniformly with $x=t\|x_A(\theta)\|_L$ for all sufficiently large $t$. Therefore, for every $\varepsilon\in(0,1)$,
\begin{equation}\label{upper-absolute}
|\mu(tA;B,C)|
\le c_5t\int_0^{2\pi}e^{-(1-\varepsilon)\phi(t\|x_A(\theta)\|_L)}\,d\theta
\le c_6t e^{-(1-\varepsilon)\phi(r(A,L)t)}.
\end{equation}
No logarithm has been taken here; we will use \eqref{upper-absolute} after proving eventual positivity.

Next, we provide the lower bound. The main task is to control the sign of $\langle \nabla \| x_A(\theta)\|_{L},\,u(\theta) \rangle$ in a neighborhood of the minimizers of $ \| x_A(\theta)\|_{L}.$ 
Differentiating $\| x_A(\theta)\|_{L}$ gives
\begin{align*}
    \frac{d}{d\theta} \| x_A(\theta)\|_{L}&= \langle \nabla \| x_A(\theta)\|_{L}, x'_A(\theta) \rangle\\&=\langle \nabla \| x_A(\theta)\|_{L}, f_A(\theta) u'(\theta)\rangle \\&=f_A(\theta)\langle \nabla \| x_A(\theta)\|_{L},  u'(\theta)\rangle.
\end{align*}
Let $\theta_0\in[0,2\pi]$ minimize $\|x_A(\theta)\|_L$. Then
$\frac{d}{d\theta}\|x_A(\theta)\|_L|_{\theta=\theta_0}=0$.
Since $f_A(\theta_0)>0$, the preceding identity gives
$\langle\nabla\|x_A(\theta_0)\|_L,u'(\theta_0)\rangle=0$.
Thus $\nabla\|x_A(\theta_0)\|_L=c\,u(\theta_0)$ for some scalar $c$.
Euler's identity for the gauge and \eqref{h} yield
\[
\|x_A(\theta_0)\|_L
=\langle\nabla\|x_A(\theta_0)\|_L,x_A(\theta_0)\rangle
=c\,h_A(\theta_0).
\]
Because both $\|x_A(\theta_0)\|_L$ and $h_A(\theta_0)$ are positive,
\begin{equation}\label{signscalar}
\big\langle\nabla\|x_A(\theta_0)\|_L,u(\theta_0)\big\rangle
=c=\frac{\|x_A(\theta_0)\|_L}{h_A(\theta_0)}>0.
\end{equation}
Since $x_A(\theta)$ parametrizes $\partial A$ and $0\in\operatorname{int}A$,
\[
\min_{\theta\in[0,2\pi]}\|x_A(\theta)\|_L
=\max\{r>0:rL\subset A\}=r(A,L).
\]

Now define the set $\Omega =\{\theta \in [0, 2\pi] :  \| x_A(\theta)\|_{L} =r(A, L)\}$.  The choice of $\theta_0$ above was arbitrary, so \eqref{signscalar} holds at every point of $\Omega$. The set $\Omega$ is compact; hence the scalar product in \eqref{signscalar} has a strictly positive minimum on $\Omega$. By continuity, there exist $\delta>0$ and $c_7>0$ such that
\[
\big\langle \nabla \| x_A(\theta)\|_{L}, u(\theta) \big\rangle > c_7\qquad\text{for all }\theta \in \Omega_\delta,
\] 
where $\Omega_\delta=\{\theta:\operatorname{dist}_{\mathbb S^1}(\theta,\Omega)<\delta\}$. Moreover, since $\mathbb S^1\setminus\Omega_\delta$ is compact and $\|x_A(\theta)\|_L>r(A,L)$ there, there exists $\gamma>0$ such that
\[
\|x_A(\theta)\|_L\ge r(A,L)+\gamma
\qquad\text{for every }\theta\in
\Omega_\delta^c:=\mathbb S^1\setminus\Omega_\delta.
\]

For brevity, set
\begin{align*}
Q_t(\theta)
&=h_B(\theta)h_C(\theta)
\Big[1-t\phi'(t\|x_A(\theta)\|_L)f_A(\theta)\\
&\qquad\times\big\langle\nabla\|x_A(\theta)\|_L,u(\theta)\big\rangle\Big]
-h'_B(\theta)h'_C(\theta).
\end{align*}

If $\Omega_\delta^c$ is empty, there is no complementary contribution. Otherwise, we show that its contribution is negligible. More precisely, as $t\to\infty$,
\begin{equation}\label{small}
e^{\phi(r(A,L)t)}
\left|\int_{\Omega_\delta^c}e^{-\phi(t\|x_A(\theta)\|_L)}
Q_t(\theta)\,d\theta\right|\longrightarrow0.
\end{equation}
Using the same absolute-value estimate as in the upper-bound argument leading to \eqref{upperbound}, it suffices to prove that
\[
c_8t\,e^{\phi(r(A,L)t)}
\int_{\Omega_\delta^c}e^{-\phi(t\|x_A(\theta)\|_L)}
\phi'(t\|x_A(\theta)\|_L)\,d\theta\longrightarrow0.
\]
For $\theta \in \Omega_\delta^c$ and $t$ large, we study 
\[
\phi(t\,\| x_A(\theta)\|_{L}) - \phi(r(A, L)\,t) - \ln  \phi'(t\,\| x_A(\theta)\|_{L}).
\]
Define
\[
\omega = \min_{\theta \in \Omega_\delta^c} \left[ 1-\frac{\| x_A(\theta)\|_{L}+r(A, L)}{2 \| x_A(\theta)\|_{L} }\right]. 
\]
Note that 
\[
1-\frac{\| x_A(\theta)\|_{L}+r(A, L)}{2 \| x_A(\theta)\|_{L} } =\frac{1}{2} - \frac{r(A, L)}{2 \| x_A(\theta)\|_{L} }\ge\frac{1}{2} - \frac{r(A, L)}{2 (r(A, L) +\gamma) } >0,
\]
so $\omega >0.$ Applying hypothesis \eqref{restr}, we may assume that  for $t$ large enough
\[
\ln  \phi'(t\,\| x_A(\theta)\|_{L}) < \omega \, \phi(t\,\| x_A(\theta)\|_{L})
\qquad\text {for all }\theta \in \Omega_\delta^c.
\] Therefore
\begin{align*}
    &\phi(t\,\| x_A(\theta)\|_{L}) - \phi(r(A, L)\,t) - \ln  \phi'(t\,\| x_A(\theta)\|_{L})\\&\ge (1-\omega)\,\phi(t\,\| x_A(\theta)\|_{L}) - \phi(r(A, L)\,t)\\ &\ge \frac{\| x_A(\theta)\|_{L}+r(A, L)}{2 \| x_A(\theta)\|_{L} }\phi(t\,\| x_A(\theta)\|_{L}) - \phi( r(A, L)\,t)\\& = \frac{r(A, L)}{ \| x_A(\theta)\|_{L} }\phi(t\,\| x_A(\theta)\|_{L}) - \phi( r(A, L)\,t)\\ &\qquad\qquad + \frac{\| x_A(\theta)\|_{L}-r(A, L)}{2 \| x_A(\theta)\|_{L} }\phi(t\| x_A(\theta)\|_{L}).
\end{align*}
By convexity of $\phi$,
\begin{align*}
&\frac{d}{dt}\left(\frac{r(A,L)}{\|x_A(\theta)\|_L}
\phi(t\|x_A(\theta)\|_L)-\phi(r(A,L)t)\right)\\
&\qquad=r(A,L)\phi'(t\|x_A(\theta)\|_L)
-r(A,L)\phi'(r(A,L)t)\ge0.
\end{align*}
Thus the expression in parentheses is nondecreasing in $t$. By compactness of
$\Omega_\delta^c$, there is a constant $c\in\mathbb R$, independent of $t$ and $\theta$, such that, for all sufficiently large $t$,
\[
\frac{r(A,L)}{\|x_A(\theta)\|_L}\phi(t\|x_A(\theta)\|_L)-\phi(r(A,L)t)\ge c
\qquad(\theta\in\Omega_\delta^c).
\]
 Moreover, using that $\| x_A(\theta)\|_{L}-r(A, L) \ge \gamma>0$ on $\Omega_\delta^c,$ we obtain 
for $t$ sufficiently large
\[
\phi(t\,\| x_A(\theta)\|_{L}) - \phi(r(A, L)\,t) - \ln  \phi'(t\,\| x_A(\theta)\|_{L}) \ge c + c_9\, \phi(t\,\| x_A(\theta)\|_{L}).
\] Thus
\begin{align}
c_8\,t \,e^{\phi(r(A, L)\,t)}  \int _{\Omega_\delta^c} e^{-\phi(t\,\| x_A(\theta)\|_{L})} & \phi'(t\,\| x_A(\theta)\|_{L})  d\theta \nonumber\\  \le & c_{10}\,t \int _{\Omega_\delta^c} e^{-c_9\,\phi(t\,\| x_A(\theta)\|_{L})}\,d\theta \nonumber \\  \le
&c_{11} \,t \,e^{-c_9\,\phi(r(A,L)\,t)} \to 0. \label{eq:complement-decay}
\end{align}
We now turn to the contribution from $\Omega_\delta$. Consider
\begin{align*}
  -  \int _{\Omega_\delta} e^{-\phi(t\,\| x_A(\theta)\|_{L})}& \Big( h_B(\theta) h_C(\theta)\\
  &\times \Big[1- \phi'(t\,\| x_A(\theta)\|_{L}) t\,f_A(\theta) \Big\langle \nabla \| x_A(\theta)\|_{L}, u(\theta) \Big\rangle \Big] \\ & - h'_B(\theta) h'_C(\theta) \Big ) \, d\theta.
\end{align*}
On $\Omega_\delta,$ by the positivity of $\Big\langle \nabla \| x_A(\theta)\|_{L}, u(\theta) \Big\rangle$ and essential boundedness of $h_B, h_B', h_C, h'_C$ and  of $|\nabla \| x_A(\theta)\|_{L}|$, we have for large $t$  
\begin{align*}
- h_B(\theta) h_C(\theta)
 \Big[1- &\phi'(t\,\| x_A(\theta)\|_{L}) t\,f_A(\theta) \Big\langle \nabla \| x_A(\theta)\|_{L}, u(\theta) \Big\rangle \Big] + h'_B(\theta) h'_C(\theta)\\
&\ge c+ c_{12} \phi'(t\| x_A(\theta)\|_{L})t\\
& \ge c_{13} \phi'(t\| x_A(\theta)\|_{L}) t,
\end{align*}
for some $c\in\mathbb R$. Here we used that $t\phi'(t\|x_A(\theta)\|_L)\to\infty$ uniformly for $\theta\in\Omega_\delta$. It therefore remains to estimate the positive integral over $\Omega_\delta$ from below.  Pick $\theta_0\in\Omega$. By periodicity, we may choose the parametrization so that
$[\theta_0,\theta_0+\delta/2]\subset\Omega_\delta$. Then
\[
\begin{aligned}
&\int _{\Omega_\delta} e^{-\phi(t\,\| x_A(\theta)\|_{L})}
  \phi'(t\,\| x_A(\theta)\|_{L})\,d\theta\\
&\qquad\ge \int _{\theta_0}^{\theta_0+\delta/2}
  e^{-\phi(t\,\| x_A(\theta)\|_{L})}
  \phi'(t\,\| x_A(\theta)\|_{L})\,d\theta.
\end{aligned}
\]
Using that $\| x_A(\theta)\|_{L}$ is a Lipschitz function, we obtain
\[
\| x_A(\theta)\|_{L} \le r(A, L) + c(\theta-\theta_0),
\]
for some $c>0$. Also, using $\phi'(t\,\| x_A(\theta)\|_{L}) \ge c_{14},$ we obtain 
\begin{equation}\label{step}
\int_{\Omega_\delta}\!\! e^{-\phi(t\| x_A(\theta)\|_{L})}\phi'(t\| x_A(\theta)\|_{L})   d\theta \ge c_{14} \int_{\theta_0}^{\theta_0+\delta/2}\!\!\! e^{-\phi(t  [r(A, L) + c(\theta-\theta_0)])} d\theta.
\end{equation}
Finally, by \eqref{restr}, for every $\varepsilon\in(0,1)$ there exists $t_\varepsilon>0$ such that, for all $t>t_\varepsilon$ and all $\theta\in[\theta_0,\theta_0+\delta/2]$,
\begin{equation}\label{eq:phi-prime-bound}
\phi'(t[ r(A, L) + c(\theta-\theta_0)]) \le e^{\varepsilon \phi(t[r(A, L) + c(\theta-\theta_0)])}.
\end{equation}
Combining \eqref{eq:phi-prime-bound} with \eqref{step}, we obtain
\begin{align}\label{eqlowe}
&\int_{\Omega_\delta}e^{-\phi(t\|x_A(\theta)\|_L)}\phi'(t\|x_A(\theta)\|_L)\,d\theta\nonumber\\
&\quad\ge c_{14}\int_{\theta_0}^{\theta_0+\delta/2}e^{-(1+\varepsilon)\phi(t[r(A,L)+c(\theta-\theta_0)])}\phi'(t[r(A,L)+c(\theta-\theta_0)])\,d\theta\nonumber\\
&\quad=\frac{c_{14}}{(1+\varepsilon)ct}\Big(e^{-(1+\varepsilon)\phi(tr(A,L))}-e^{-(1+\varepsilon)\phi(t(r(A,L)+c\delta/2))}\Big)\nonumber\\
&\quad\ge\frac{c_{15}}t e^{-(1+\varepsilon)\phi(tr(A,L))},
\end{align}
where the last inequality holds for large $t$ because
$\phi(t(r(A,L)+c\delta/2))-\phi(tr(A,L))\to\infty$. 
Combining \eqref{eqlowe} with \eqref{eq:complement-decay}, and choosing
$0<\varepsilon<c_9$, so that the contribution of $\Omega_\delta^c$ has smaller exponential order, we obtain, for all sufficiently large $t$,
\begin{align*}
-\mu(tA;B,C)
&=-\int_0^{2\pi}e^{-\phi(t\|x_A(\theta)\|_L)}Q_t(\theta)\,d\theta\\
&\ge c_{16}e^{-(1+\varepsilon)\phi(r(A,L)t)}.
\end{align*}
The preceding estimate proves, in particular, that $-\mu(tA;B,C)>0$ for all sufficiently large $t$. We may therefore take logarithms. For every $0<\varepsilon<c_9$ it gives
\[
\liminf_{t\to\infty}
\frac{\ln[-\mu(tA;B,C)]}{\phi(r(A,L)t)}\ge-(1+\varepsilon).
\]
On the other hand, \eqref{upper-absolute} and \eqref{eq:lim} imply
\[
\limsup_{t\to\infty}
\frac{\ln[-\mu(tA;B,C)]}{\phi(r(A,L)t)}\le-(1-\varepsilon)
\qquad(0<\varepsilon<1).
\]
Letting $\varepsilon\downarrow0$ proves the asserted limit.

This completes the proof. 
\end{proof}

\begin{remark}
The ball example explains the role of hypothesis~\eqref{restr}. Indeed, if
$L=A=B=C=B_2^2$, then
\[
\mu(tA;B,C)=2\pi e^{-\phi(t)}\bigl(1-t\phi'(t)\bigr),
\]
which is asymptotic to $-2\pi t\phi'(t)e^{-\phi(t)}$ whenever $t\phi'(t)\to\infty$. Thus, without an assumption controlling $\ln\phi'(t)$ relative to $\phi(t)$, the additional term $\ln t+\ln\phi'(t)$ may contribute to the logarithmic rate. This shows that a condition of the form \eqref{restr} is natural for the stated conclusion.
\end{remark}

\begin{remark}
The higher-dimensional analogue remains open. The main obstacle is that a direct application of \eqref{defmeas} does not allow us to control the sign of $\langle\nabla\rho(n_A^{-1}(u)),\nabla h_B(u)\rangle$ near a tangency point of $A$ and $r(A,L)L$, unlike the argument leading to \eqref{signscalar}. The difficulty arises because this scalar product also depends on $B$.
\end{remark}

\end{document}